\documentclass{amsproc}
\usepackage{eurosym}
\usepackage{amssymb}
\usepackage{amsfonts}
\usepackage[utf8]{inputenc}
\usepackage[OT4]{fontenc}

\def\cal{\mathcal}

\newtheorem{defi}{Definition}[section]

\newtheorem{theo}[defi]{Theorem}
\newtheorem{co}[defi]{Corollary}
\newtheorem{pr}[defi]{Proposition}
\newtheorem{re}[defi]{Remark}

\font\tenmsy=msbm10

\def\Bbb#1{\hbox{\tenmsy#1}} 

\address[Z. Jelonek]{Instytut Matematyki, Polska Akademia Nauk, \'Sw. Tomasza 30, 31-027 Krak\'ow, Poland} 
\email{najelone@cyf-kr.edu.pl}

\address[M. Laso\'n]{Instytut Matematyki, Polska Akademia Nauk, \'Sw. Tomasza 30, 31-027 Krak\'ow, Poland;
Wydzia{\l} Matematyki i Informatyki, Uniwersytet Jagiello\'nski, Prof. S. \L ojasiewicza 6, 30-348 Krak\'ow, Poland}
\email{michalason@gmail.com}

\thanks{The first author was partially supported by the grant of Polish Ministry of Science, 2006-2009}

\begin{document}

\title[The set of fixed points]{The set of fixed points of a unipotent group}

\begin{abstract}
Let $K$ be an algebraically closed field. Let $G$ be a non-trivial
connected unipotent group, which acts effectively on an affine
variety $X.$ Then every non-empty component $R$ of the set of
fixed points of $G$ is a $K$-uniruled variety, i.e, there exists
an affine cylinder $W\times K$ and a dominant, generically-finite
polynomial mapping $\phi : W\times K\rightarrow R.$ We show also
that if an arbitrary infinite algebraic group $G$ acts effectively
on $K^n$ and the set of fixed points contains a hypersurface $H,$
then this hypersurface is $K$-uniruled.
\end{abstract}

\maketitle

\section{Introduction.}
Let $K$ be an algebraically closed field (of arbitrary characteristic). Let $G$ be a connected unipotent
algebraic group, which acts effectively on a variety $X.$ The set
of fixed points of this action was studied intensively (see e.g.,
\cite{b-b},  \cite{car}, \cite{gross}, \cite{hor}). In particular
Bia\l ynicki-Birula has proved that if $X$ is an affine variety,
then $G$ has no isolated fixed points.

Here we consider the case when $X$ is an affine variety. We
generalize the result of Bia\l ynicki-Birula and we prove, that
the set $Fix(G)$ of fixed points of $G$ is in fact a $K$-uniruled
variety. In particular for every  point $x\in Fix(G)$ there is a
polynomial curve $\phi : K\to Fix(G)$ such that $\phi(0)=x.$

We show also that if an arbitrary infinite algebraic group $G$
acts effectively on $K^n$ and the set of fixed points contains a
hypersurface $H,$ then this hypersurface is $K$-uniruled. This
generalizes \cite{jel}.

\section{ Preliminaries.}
At the begining  we recall some basic facts about $K$-uniruled
varieties (see \cite{jel?}).

\begin{pr} Let $\Gamma$ be an affine curve. The following two statements
are equivalent:

1) there exists a regular birational map $\phi: K \rightarrow
\Gamma$;

2) there exists a regular dominant map $\phi:K \rightarrow
\Gamma$.
\end{pr}

\begin{defi} Let $\Gamma$ be an affine curve which has the property 1)
(or 2)) from the above proposition. Then $\Gamma$ will be called
\emph{an affine polynomial curve} and the mapping $\phi$ will be
called \emph{a parametrization of $\Gamma$}. A family $\cal F$ of
affine polynomial curves on $X$ is called \emph{bounded} if there
exists an embedding $i: X\subset K^N$ and a natural number $D$
such that every curve $\Gamma\in {\cal F}$ has degree less then or
equal to $D$ in $K^N.$
\end{defi}

\begin{re}
It is easy to see that the definition of bounded family does
not depend on an embedding $i: X\to K^N.$
\end{re}

Now we give the definition of a {\it $K$-uniruled} variety. We
have introduced this notion for uncountable fields in \cite{jel?}.
However, here we work over any field and we need a refined version
of the definition (it coincides with the older one for uncountable
fields).

\begin{pr}\label{family}
Let $X\subset K^N$ be an irreducible affine variety of dimension
$\geq 1.$ The following conditions  are equivalent:

1) there is a bounded family $\cal F$ of affine polynomial curves,
such that for every point $x\in X$ there is a curve $l_x\in \cal
F$ going through x;

2) there is an open, non-empty subset $U\subset X$ and a bounded
family $\cal F$ of affine polynomial curves, such that for every
point $x\in U$ there is a curve $l_x\in \cal F$  going through x;

3) there exists an affine variety $W$ with {\rm dim} $W = {\rm
dim}\ X-1$ and a  dominant polynomial mapping $\phi: W\times
K\rightarrow X.$
\end{pr}

\begin{proof} 1) $\Rightarrow$ 2) is obvious. 2) $\Rightarrow$ 3)
follows from \cite{stasica}.  3) $\Rightarrow$ 2) is obvious. We
prove 2) $\Rightarrow$ 1). Assume that $X\subset K^n.$
Every  curve
$l_x\in \cal F$ is given by $n$ polynomials of one variable:
$$l_x(t)=(x_1+\sum_{i=1}^D a^{1,i} t^i,...,x_n+\sum_{i=1}^D a^{n,i}
t^i).$$ Let $\Delta$ denote a $nD-1$-dimensional weighted
projective space with weights $1,2,...,D,$ $...,1,2,...,D.$ Hence
we can associate $l_x$ with one point
$$(x_1,...,x_n; a^{1,1},...,a^{1,D};a^{2,1},...,a^{2,D};...;a^{n,1},...,a^{n,D})\in
X\times \Delta.$$ Let $\{f_i=0, \ i=1,...,m\}$ ($f_i\in
K[x_1,...,x_n]$) be equations of the variety $X.$ The condition
$l_x\subset X$ is equivalent to conditions $f_i(l_x(t))=0, \ i
=1,..., m.$ The last equations are in fact equivalent to a finite
number of polynomial equations
$$h_\alpha(x_1,..., x_n; a^{1,1},...,a^{1,D};a^{2,1},...,a^{2,D};...;a^{n,1},...,a^{n,D})=0,$$
which are weighted homogeneous with respect to
$a^{1,1},...,a^{1,D};a^{2,1},...;a^{n,1},...,a^{n,D}.$
Let $W\subset X\times \Delta$ be a variety described by
polynomials $h_\alpha$ and let $\pi : X\times \Delta\to X$ be the
projection. The mapping $\pi$ is proper, in particular the set
$\pi(W)$ is closed. Since $\pi(W)$  contains the dense subset $U$,
we have $\pi(W)=X.$
\end{proof}

Now we can state:

\begin{defi}
An affine irreducible variety $X$ is called \emph{$K$-uniruled}
if it is of dimension $\geq 1$, and satisfies one of
equivalent conditions $1)-3)$ listed in Proposition \ref{family}.
\end{defi}

If the field $K$ is uncountable we have stronger result (see
\cite{stasica}):

\begin{pr}\label{stasica}
Let $K$ be an uncountable field. Let $X$ be an irreducible affine
variety of dimension $\geq 1.$ The following conditions  are
equivalent:

1) $X$ is $K$-uniruled;

2) for every point $x\in X$ there is a polynomial affine curve in
$X$ going through x;

3) there exists a Zariski-open, non-empty subset $U$ of $X$, such
that for every point $x\in U$ there is a polynomial affine curve
in $X$ going through x;

4) there exists an affine variety $W$ with {\rm dim} $W = {\rm
dim}\ X-1$ and a  dominant polynomial mapping $\phi: W\times
K\rightarrow X.$
\end{pr}

Let $X$ be a smooth projective surface and let $D=\sum_{i=1}^n
D_i$ be a simple normal crossing (s.n.c) divisor on $X$ (here we
consider only reduced divisors). Let $graph(D)$ be a graph of $D$,
i.e., a graph with one vertex $Q_i$ for each irreducible component
$D_i$  of $D$, and one edge between $Q_i$ and $Q_j$ for each point
of intersection of $D_i$ and $D_j$.

\begin{defi}
Let $D$ be a simple normal crossing divisor on a smooth surface
$X.$ We say that $D$ is a \emph{tree} if $graph(D)$ is connected and
acyclic.
\end{defi}

We have the following fact which is obvious from graph theory:

\begin{pr}\label{acykl}
Let $X$ be a smooth projective surface and let  divisor $D\subset
X$ be a tree. Assume that $D', D''\subset D$ are connected
divisors without common components. Then  $D'$ and  $D''$ have at
most one common point.
\end{pr}

\begin{defi}
Let $X, Y$ be affine varieties and $f : X \rightarrow Y$ be a
regular  mapping. We say that $f$ \emph{is finite} at a point $y
\in Y$ if there exists an open neighborhood $U$ of  $y$ such that
$ {\rm res}_{f^{-1}(U)} f :f^{-1}  (U)\rightarrow U$ is a finite
map.
\end{defi}

It is well-known that if $f$ is a generically finite, then the set
of points at which $f$ is not finite is either empty or it is a
hypersurface in $\overline{f(X)}$ (for details see \cite{jel?},
\cite{jel2}). We denote this set by $S_f.$ Now we can formulate
the following useful:

\begin{theo}\label{gl1}
Let $\Gamma$ be an affine curve. Let $\phi : \Gamma\times K\to
K^N$ be a generically finite mapping. Then the set $S_{\phi}$ is a
union of finitely many (possibly empty) of affine polynomial
curves.
\end{theo}

\begin{proof}
Taking a normalization we can assume that the curve $\Gamma$ is
smooth (note that a normalization is a finite mapping). Let
$\overline{\Gamma}$ be a smooth completion of $\Gamma$ and denote
$\overline{\Gamma}\setminus\Gamma=\{ a_1,..., a_l\}.$ Let $X=
\Gamma\times K$ and $\overline{X}= \overline{\Gamma}\times \Bbb
P^1$ be a projective closure of $X.$ The divisor
$D=\overline{X}\setminus X=\overline{\Gamma}\times \infty +
\sum_{i=1}^l \{a_i\}\times \Bbb P^1$ is a tree. Now we can resolve
points of indeterminacy of the mapping $\phi$:
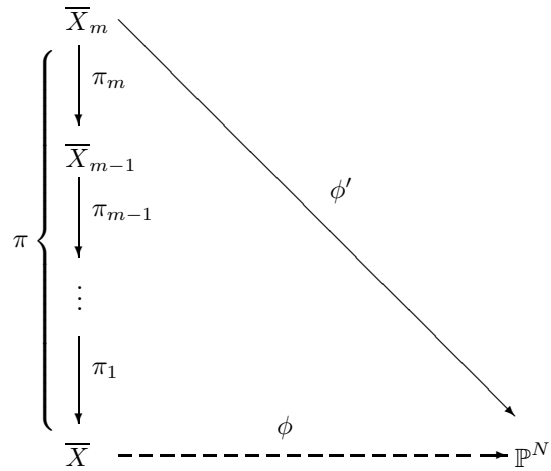
\vspace{3mm}
\begin{center}\begin{picture}(240,160)(-40,40)
\put(-20,117.5){\makebox(0,0)[l]{$\pi\left\{\rule{0mm}{2.7cm}\right.$}}
\put(0,205){\makebox(0,0)[tl]{$\overline{X}_m$}}
\put(0,153){\makebox(0,0)[tl]{$\overline{X}_{m-1}$}}
\put(4,105){\makebox(0,0)[tl]{$\vdots$}}
\put(0,40){\makebox(0,0)[tl]{$\overline{X}$}}
\put(170,40){\makebox(0,0)[tl]{$\Bbb P^N$}}
\put(80,50){\makebox(0,0)[tl]{$\phi$}}
\put(100,140){\makebox(0,0)[tl]{$\phi '$}}
\put(10,70){\makebox(0,0)[tl]{$\pi_1$}}
\put(10,130){\makebox(0,0)[tl]{$\pi_{m-1}$}}
\put(10,180){\makebox(0,0)[tl]{$\pi_m$}}
\put(5,190){\vector(0,-1){30}} \put(5,140){\vector(0,-1){30}}
\put(5,80){\vector(0,-1){33}}
\multiput(20,35)(8,0){17}{\line(1,0){5}}
\put(157,35){\vector(1,0){10}} \put(20,200){\vector(1,-1){150}}
\end{picture}
\end{center}
\vspace{3mm} Note that  the divisor $D'=\pi^*(D)$ is a tree. Let
$\overline{\Gamma}\times\infty'$ denote a proper transform of
$\overline{\Gamma}\times\infty.$ It is an easy observation that
$\phi'(\overline{\Gamma}\times\infty')\subset H_\infty$, where
$H_\infty$ denotes the hyperplane at infinity of $\Bbb P^N.$ Now
$S_\phi= \phi'(D'\setminus \phi'^{-1}(H_\infty)).$ The curve
$L=\phi'^{-1}(H_\infty)$ is a complement of a semi-affine variety
$\phi'^{-1}(K^N)$ hence it is connected (for details see
\cite{jel?}, Lemma~4.5). Now by Proposition \ref{acykl} we have
that every irreducible curve $Z\subset D'$ which does not belong
to $L$ has at most one common point with $L$. Let $S\subset
S_\phi$ be an irreducible component. Hence $S$ is a curve. There
is a curve $Z\subset D',$ which has exactly one common point with
$L$ such that $S=\phi'(Z\setminus L)=\phi'(K)$. This completes the
proof.
\end{proof}

\section{Main Result.}
The aim of this section is to prove the following:

\begin{theo}\label{glowne}
Let $G$ be a non-trivial connected unipotent  group which acts
effectively on an affine variety $X.$ Then every non-empty
component $R$ of the set of fixed points of $G$ is a $K$-uniruled
variety.
\end{theo}

\begin{proof} First of all let us recall that a connected unipotent group
has a normal series $$0=G_0\subset G_1\subset ...\subset G_r=G,$$
where $G_i/G_{i-1}\cong G_a= (K,+,0).$ By induction on dim $G$ we
can easily reduce the general case to that of $G=G_a.$

First assume that the field $K$ is uncountable. Take a point $a
\in R.$ By Proposition \ref{stasica} it is enough to prove, that
there is an affine polynomial curve $S\subset Fix(G)$ through $a.$
Let $L$ be an irreducible curve in $X$ going through $a,$ which is
not contained in any orbit of $G$ and it is not contained in
$Fix(G).$ Consider a surface $Y=L\times G.$ There is natural $G$
action on $Y$ :  for $h\in G$ and $y=(l,g)\in Y$ we put
$h(y)=(l,hg)\in Y.$ Take a mapping
$$\Phi : L\times G\ni (x, g)\to g(x)\in X.$$ It is a generically-finite
polynomial mapping. Observe that it is $G$-invariant i.e.,
$\Phi(gy)=g\Phi(y).$ This implies that the set $S_\Phi$ of points
at which the mapping $\Phi$ is not finite is $G$-invariant.
Indeed, it is enough to show that the complement of this set is
$G$-invariant. Let $\Phi$ be finite at $x\in X.$ This means that
there is an open neighborhood $U$ of $x$ such that the mapping
$\Phi : \Phi^{-1}(U)\to U$ is finite. Now we have the following
diagram:
\begin{center}
\begin{picture}(280,130)(-40,40)
\put(140,160){\makebox(0,0)[tl]{$\Phi^{-1}(gU)=g\Phi^{-1}(U)$}}
\put(20,160){\makebox(0,0)[tl]{$\Phi^{-1}(U)$}}
\put(30,40){\makebox(0,0)[tl]{$U$}}
\put(180,40){\makebox(0,0)[tl]{$gU$}}
\put(190,100){\makebox(0,0)[tl]{$\Phi$}}
\put(40,100){\makebox(0,0)[tl]{$\Phi$}}
\put(95,50){\makebox(0,0)[tl]{$g$}}
\put(95,170){\makebox(0,0)[tl]{$g$}}
\put(33,145){\vector(0,-1){100}} \put(45,35){\vector(1,0){125}}
\put(60,155){\vector(1,0){75}} \put(183,145){\vector(0,-1){100}}
\end{picture}
\end{center}
\vspace{3mm}

\noindent This diagram shows that the mapping $\Phi$ is finite
over $gU$ if it is finite over $U.$ In particular this implies
that the set $S_\Phi$ is $G$-invariant. Let  $S_\Phi=S_1\cup
S_2...\cup S_k$ be a decomposition of $S_\Phi$ in (irreducible)
affine polynomial curves (see Theorem \ref{gl1}). Since the set
$S_\Phi$ is $G$-invariant, we have that each curve $S_i$ is also
$G$-invariant. Note that the point $a$ belongs to $S_\Phi$,
because the fiber over $a$ has infinite number of points. We can
assume that $a\in S_1.$ Let $x\in S_1,$ we want to show that $x\in
Fix(G).$ Indeed, otherwise $G.x=S_1$ and $a$ would be in the orbit
of $x$ - a contradiction. Hence $S_1\subset Fix(G)$ and we
conclude our result by Theorem \ref{gl1}.

Now assume that the field $K$ is countable. Let $X\subset K^n.$
Let $T$ be uncountable algebraically closed extension of $K.$ By
the base change the group $G$ acts on $\overline{X}\subset T^n.$
Moreover, the variety $\overline{R}\subset T^n$ is a component of
the set of fixed points of $G$ (because the set $R$ is dense in
$\overline{R}$). By the first part of our proof the variety
$\overline{R}$ is $T$-uniruled. In particular there exists a
number $D$  such that for every point $x\in \overline{R}$ there is
a polynomial affine curve $l_x\subset \overline{R}\subset T^n,$ of
degree at most $D,$ going through $x$.
  Note that it is true for every point  $x\in R.$

Every such curve  $l_x$ is given by $n$ polynomials of one
variable: $$l_x(t)=(x_1+\sum_{i=1}^d a^{1,i}
t^i,...,x_n+\sum_{i=1}^d a^{n,i} t^i),$$ where $d\le D.$ Hence we
can associate $l$ with one point
$$(a^{1,0},a^{1,1},...,a^{1,d};a^{2,0},...,a^{2,d};...;a^{n,0},...,a^{n,d})\in
T^n.$$ We can assume without loss of generality that $a^1_d=1.$
Let $\{f_i=0, \ i=1,...,m\}$ ($f_i\in K[x_1,...,x_n]$) be
equations of the variety $S.$ The condition $l_x\subset
\overline{R}$ is equivalent to conditions $f_i(l(t))=0, \ i
=1,..., m.$ The last equations are in fact equivalent to a finite
number of polynomial equations
$$h_\alpha(a^{1,0},a^{1,1},...,a^{1,d};a^{2,0},...,a^{2,d};...;a^{n,0},...,a^{n,d})=0,$$
where $h_\alpha \in K[y_1,...,y_N].$ Equations $h_\alpha=0$ plus
extra conditions $a^{i,0}=x_i, \ i =1,...,n$ and $a^{1,d}=1$ have
solutions in the field $T$, hence they have also  solutions in the
field $K.$

  This means that we can find an affine polynomial
curve $l_x$ over the field $K$ of degree at most $D,$ which  is
contained in $R$ and goes through $x.$  Consequently the variety
$R$ is $K$-uniruled. The proof of Theorem \ref{glowne} is
complete.\end{proof}

\begin{co}(Bia\l ynicki-Birula, \cite{b-b}).
Let $G$ be a non-trivial connected unipotent  group which acts
effectively on an affine variety $X.$ Then $G$ has no isolated
fixed points.
\end{co}

Theorem \ref{glowne} (or rather its proof) suggests  a following
generalization of \cite{jel}:

\begin{theo}\label{glowne1}
Let $G$ be an infinite connected algebraic group which acts
effectively on $K^n,$ $\ n\ge 2.$ Assume that an irreducible
hypersurface $W$ is contained in the set of fixed points of $G.$
Then $W$ is $K$- uniruled.
\end{theo}

\begin{proof} Since $G$ acts effectively on affine space $K^n$ we can
assume by the Chevalley Theorem (see \cite{sha}, Th. C, p.190)
that the group $G$ is affine. In particular it contains either the
subgroup $G_m=(K^*,\cdot, 1)$ or the subgroup $G_a=(K,+,0)$ (see
e.g., \cite{iit}). Thus we can assume that $G$ is either $G_m$ or
it is $G_a.$

As before we can assume that the field $K$ is uncountable. Take a
point $a \in W.$ By Proposition \ref{stasica} it is enough to
prove, that there is an affine parametric curve $S\subset W$
through $a.$ Let $L$ be a line in $K^n$ going through $a$ such
that the set $L\cap Fix(G)$ is finite. Set $L\cap W=\{ a, a_1,...,
a_m\}.$ Now consider a mapping
$$\phi : L\times G\ni (x, g)\to g(x)\in K^n.$$ Observe that
$\phi(L\times G)$ is a union of disjoint orbits of $G.$ This
implies $\phi(L\times G)\cap W=\{ a, a_1,..., a_m\}.$ Take
$X=\overline{\phi(L\times G)}.$ Note that $X\cap W$ is a union of
curves. This means that there is a curve $S\subset X\cap W$, which
contains the point $a.$ However $S\subset \overline{X\setminus
\phi(L\times G)}.$ This implies that $S\subset S_\phi$ and we
conclude by Theorem \ref{gl1}.
\end{proof}

To finish this note we state:

\vspace{2mm} \noindent{\bf Conjecture.} {\it Let $K$ be an
algebraically closed field. Let $G$ be an algebraic group, which
acts effectively on $K^n.$ If $S$ is an irreducible component of
the set of fixed points of $G$, then $S$ is either a point or it
is a $K$-uniruled variety.}

\section*{Acknowledgments} 
We are grateful to Andrzej Bia\l ynicki-Birula  for helpful discussions. We also want to thank
Tadeusz Mostowski  for helpful remarks.

\end{document}